\documentclass[12pt]{amsart}

\usepackage{amsxtra,amssymb,amsthm,amsmath,amscd,url,listings}
\usepackage[margin=1in,footskip=0.5in]{geometry}
\usepackage[utf8]{inputenc}
\usepackage{eucal}
\usepackage{fullpage,enumerate,enumitem}
\usepackage{scrtime}
\usepackage{supertabular}

\usepackage[colorlinks]{hyperref}
\usepackage{graphicx}

\theoremstyle{plain}
\newtheorem{theorem}{Theorem}[section]
\newtheorem{lemma}[theorem]{Lemma}
\newtheorem{corollary}[theorem]{Corollary}
\newtheorem{conjecture}[theorem]{Conjecture}

\theoremstyle{remark}

\theoremstyle{definition}

\newtheorem{definition}[theorem]{Definition}
\newtheorem{assumption}[theorem]{Assumption}

\newtheorem{remark}[theorem]{Remark}

\newcommand{\N}{{\mathbb N}}

\newcommand{\R}{{\mathbb R}}
\newcommand{\Z}{{\mathbb Z}}
\newcommand{\F}{{\mathbb F}}

\renewcommand{\a}{{\alpha}}

\newcommand{\ep}{\epsilon}

\title{On pseudo-polynomials divisible only by a sparse set of primes and $\a$-primary pseudo-polynomials}
\author{Vivian Kuperberg}
\address{Department of Mathematics, Stanford University, Stanford, CA, USA}
\email{viviank@stanford.edu}

\thanks{The author is supported by NSF GRFP grant DGE-1656518, and would like to thank Kevin Carde, Emmanuel Kowalski, Kannan Soundararajan, Avi Wigderson, and Yuval Wigderson for helpful comments and discussions.}

\begin{document}

\begin{abstract}
We explore two questions about pseudo-polynomials, which are functions $f:\N \to \Z$ such that $k$ divides $f(n+k) - f(n)$ for all $n,k$. First, for certain arbitrarily sparse sets $R$, we construct pseudo-polynomials $f$ with $p|f(n)$ for some $n$ only if $p \in R$. This implies that not all pseudo-polynomials satisfy an assumption of a recent paper of Kowalski and Soundararajan. We also consider $\a$-primary pseudo-polynomials, where the pseudo-polynomial condition is only required for $k$ lying in a set of primes of density $\a$. We show that if an $\a$-primary pseudo-polynomial is $O(e^{(2/3-\ep) n})$, then it is a polynomial.
\end{abstract}
\maketitle

\section{Introduction}

A \emph{pseudo-polynomial}, as studied by Hall in \cite{MR294300}, is a function $f: \N \to \Z$ such that for all $m,n \in \N$, $|m-n|$ divides $f(m) - f(n)$. All polynomials with integer coefficients are pseudo-polynomials, but the class of pseudo-polynomials is much larger. Hall showed in \cite{MR294300} that $f$ is a pseudo-polynomial if and only if it can be expressed as
\[f(n) = A_0 + \frac{A_1}{1!} \cdot n + \frac{A_2}{2!} \cdot n \cdot (n-1) + \frac{A_3}{3!} \cdot n \cdot (n-1) \cdot (n-2) + \cdots, \]
where $\mathrm{lcm}[1,\dots,m]|A_m$ for all $m$. Thus the class of pseudo-polynomials is large; there are uncountably many examples. Two classical examples to consider are $f_1(n) = \lfloor en! \rfloor$ (in which case $A_m = m!$ for all $m$) and $f_2(n) = (-1)^nD(n) = 1-n + n(n-1) + \cdots + (-1)^n n!$, an alternating sum of the number $D(n)$ of permutations on $n$ elements with no fixed points.

Pseudo-polynomials can also be seen through the lens of $p$-adic continuity. Let $\|\cdot\|_p$ be the $p$-adic valuation on the integers, so that $\|n\|_p = \frac 1{p^k}$ where $p^k\parallel n$. Then $f$ is a pseudo-polynomial if and only if for all $p$ and for all $m,n \in \N$, $\|f(m)-f(n)\|_p \le \|m-n\|_p$. In other words, for all primes $p$, $f$ is $p$-adically Lipschitz with constant $1$. 

If $f$ does not grow very quickly as well as satisfying these $p$-adic constraints, $f$ is often known or conjectured to be a polynomial. Hall \cite{MR294300} and Ruzsa \cite{MR323688} independently showed that if a pseudo-polynomial $f(n)$ is $O((e-1)^n)$, then $f$ must be a polynomial. Ruzsa conjectured that the same should be true if $f(n)$ is $O(c^n)$ for some $c<e$; meanwhile Hall provided an example of a pseudo-polynomial that is $O(e^n)$ but not a polynomial. Zannier \cite{MR1397665} has the best known result in this direction, showing that if $f(n)$ is $O((e^{0.75})^n)$, then $f$ must be a polynomial. Meanwhile, Bell and Nguyen \cite{MR4158738} recently proved an analogue of Ruzsa's conjecture in the function field case. Other recent work (\cite{delayguerivoal}, \cite{odesky}, \cite{odeskyrichman}) has provided new perspectives and generalizations of these problems.

In this paper, we consider two relaxations of the constraints on $f$. First, we study the behavior of a general pseudo-polynomial, without growth constriants, and construct poorly-behaved pseudo-polynomials that act very differently from polynomials. Second, we relax the $p$-adically $1$-Lipschitz constraint to hold only for a subset of primes with positive density, rather than for all primes $p$. We then consider growth constraints for these functions which require them to be polynomials.

In Section \ref{sec:crtmain}, we construct fast-growing, poorly-behaved pseudo-polynomials. Pseudo-polynomials are precisely those functions $f: \N \to \Z$ such that for every modulus $q$, $f$ restricts to a well-defined map from $\Z/q \to \Z/q$. Many properties of polynomials can be phrased in terms of these reductions mod $q$, so it is natural to ask if these properties hold for all pseudo-polynomials as well. We provide a family of examples of poorly-behaved pseudo-polynomials whose mod-$q$ behavior differs wildly from that of polynomials. We construct a pseudo-polynomial $f$ one value at a time, picking $f(n)$ based on the values $f(0), \dots, f(n-1)$ that we have already chosen. As we'll see, the pseudo-polynomial condition will still give us the leeway to pick values of $f(n)$ that resist polynomial properties. 

We now introduce the specific polynomial property under consideration. In Kowalski and Soundararajan's recent paper \cite{MR4252763}, they show equidistribution of certain subsets of $\R/\Z$. In their setting, for each prime $p$ they fix a set $A_p$ of residue classes modulo $p$, which they then use to construct subsets of $\R/\Z$. In establishing their Theorem 1.2, they impose the following assumption on the input sets $A_p$.

\begin{assumption}[\cite{MR4252763}]\label{assum:sk}
There exist constants $\a > 0$ and $x_0 \ge 2$ such that for all $x \ge x_0$,
\[\sum_{\substack{p \le x \\ |A_p| \ge 1}} \log p \ge \a x.\]
\end{assumption}

For a polynomial $f \in \Z[X]$, take $A_p$ to be the set of roots of $f$ modulo $p$. In this case, Assumption \ref{assum:sk} says that the set of primes $p$ for which $f$ has at least one root modulo $p$ is not too sparse; the Chebotarev density theorem implies that this must be true for any polynomial $f$. Kowalski and Soundararajan discuss generalizing this property to pseudo-polynomials. We will say that a pseudo-polynomial $f$ \emph{satisfies Assumption \ref{assum:sk}} if the sets $A_p = \{n \mod p: f(n) \equiv 0 \mod p\}$ satisfy Assumption \ref{assum:sk}. Kowalski and Soundararajan provide numerical evidence that the pseudo-polynomials $f_1(n) = \lfloor en! \rfloor$ and $f_2(n) = (-1)^n D(n)$ satisfy Assumption \ref{assum:sk}, and wonder if the same is true of all pseudo-polynomials. We will construct an uncountable family of pseudo-polynomials that do not satisfy Assumption \ref{assum:sk}. Our construction is as follows:

\begin{theorem} \label{main}
Let $\mathcal R \subseteq \N$ be any set of primes such that for all $Q \in \N$ and for all congruence classes $a \pmod Q$ with $(a,Q) = 1$, there exists $r \in \mathcal R$ with $r \equiv a \pmod Q$. Then there exists a pseudo-polynomial $f: \N \to \Z$ such that for all $n$ and for all primes $r$, $r|f(n)$ only if $r \in \mathcal R$. 
\end{theorem}

The requirement on $\mathcal R$ here is very flexible, which lets $\mathcal R$ be arbitrarily sparse in the following sense.

\begin{corollary}\label{maincor}
Let $b:\N \to \N$ be any function. There exists a pair $(\mathcal R,f)$ where $\mathcal R = \{r_k\}_{k \in \N}$ is a set of primes such that $r_k > b(k)$ for all $k$, and $f: \N \to \Z$ is a pseudo-polynomial such that for all $n$ and for all primes $r$, $r|f(n)$ only if $r \in \mathcal R$.
\end{corollary}

If $b$ grows sufficiently quickly, $f$ will not satisfy Assumption \ref{assum:sk}. For example, let $b(n) = e^n$, let $R$ and $f$ be as constructed in Corollary \ref{maincor}. Crucially, $|A_p| \ge 1$ only if $p \in \mathcal R$. Thus
\[\sum_{\substack{p \le x \\ |A_p| \ge 1}} \log p < (\log x) \cdot \log x = O(\log x)^2,\]
since there are at most $\log x$ terms in the sum and each is of size at most $\log x$. Here Assumption \ref{assum:sk} is not satisfied.

\begin{remark}
The construction in Theorem \ref{main} can be modified to construct pseudo-polynomials $f$ that satisfy Assumption \ref{assum:sk}, but for which $|A_p| \ge 2$ very rarely. This can be done by considering two disjoint sets of primes $\mathcal R$ and $\mathcal S$, with $\mathcal R$ extremely sparse. Then, one can specify that $f$ has exactly one root modulo $p$ for $p \in \mathcal S$, and construct $f$ very analogously using primes in $\mathcal R$. For these pseudo-polynomials, the Theorem 1.2 of \cite{MR4252763} would apply, but its conclusion is fairly weak.
\end{remark}

\begin{remark}
The pseudo-polynomials that we construct here to fail Assumption \ref{assum:sk} grow very quickly, much faster than $f_1$ or $f_2$. It seems plausible not only that $f_1$ and $f_2$ satisfy Assumption \ref{assum:sk}, but that any pseudo-polynomial that fails Assumption \ref{assum:sk} must grow much faster.
\end{remark}

In a different direction, we also explore the following relaxation of the pseudo-polynomial definition. 
\begin{definition}
Let $0 < \alpha < 1$ and let $\mathcal P$ be a set of primes with $\prod_{\substack{p \in \mathcal P \\ p \le j}} p = e^{\alpha j} + o(e^{\alpha j})$. A \emph{$\alpha$-primary pseudo-polynomial} is a function $f:\N \to \Z$ such that for all $p \in \mathcal P$, for all $n \in \N$,  
\[f(n + p) \equiv f(n) \pmod p. \]
\end{definition}

Note that we are only considering constraints on the behavior of $f \pmod p$ for $p$ prime, rather than prime powers. This is strictly weaker than requiring that $f$ be $p$-adically $1$-Lipschitz for all $p \in \mathcal P$, but strong enough for our purposes. 

For $j \in \N$, we will denote by $\Pi(j)$ the quantity
\[\Pi(j) := \prod_{\substack{p \in \mathcal P \\ p \le j}} p,\]
where $\Pi(j)=e^{\a j}+ o(e^{\a j})$. Then
\[f(n) = \sum_{j=0}^\infty \Pi(j)\frac{n(n-1)\cdots(n-j)}{j!}\]
is an $\a$-primary pseudo-polynomial, with $f(n) = O(e^{(\a + \ep)n})$ and $f$ is not a polynomial. We conjecture the following analogue of the Ruzsa's conjecture.

\begin{conjecture}
Let $f$ be an $\a$-primary pseudo-polynomial with $f(n) = O(c^n)$ for $c < e^{\a}$. Then $f$ is a polynomial.
\end{conjecture}

In Section \ref{sec:zannierrho}, we use techniques of Zannier \cite{MR1397665} to show partial progress towards this conjecture for all $\a$. 

\begin{theorem}\label{thm:zannierrho}
Fix $0 < \a \le 1$ and $\ep > 0$. Let $f$ be a $\alpha$-primary pseudo-polynomial such that
\[|f(n)| = O(e^{(2/3 - \ep)\a n}).\]
Then $f$ is a polynomial.
\end{theorem}
When $\alpha = 1$, the constant $2/3$ is weaker than the constant $3/4$ in the result of Zannier \cite{MR1397665}. However, Theorem \ref{thm:zannierrho} has a uniform dependence on the density $\alpha$, giving a nontrivial bound for any $\alpha > 0$. Our proof closely follows Zannier's, but is slightly simpler and thus easier to generalize.

\section{Proofs of Theorem \ref{main} and Corollary \ref{maincor}}\label{sec:crtmain}

\begin{proof}[Proof of Theorem \ref{main}]
Our strategy will be as follows. We construct $f$ by inductively picking values for $f(n)$ such that for all $m < n$, $n-m$ divides $f(n) - f(m)$. Each value $f(n)$ will be constructed as a product of primes in $\mathcal R$, so we first construct $\mathcal R$ carefully. Our assumption on $\mathcal R$, that it contains an element of every reduced congruence class modulo $Q$ for all $Q \in \N$, lets us guarantee that the constraints for $f(n)$ are satisfiable by a product of elements of $\mathcal R$.

Enumerate $\mathcal R$ as $\mathcal R = \{r_k\}_{k \in \N}$. Set $f(0) = r_0$ and $f(1) = r_1$. Assume that $f(m)$ has been defined for all $m < n$; we will choose $f(n)$ so that $(n-m)|(f(n) - f(m))$ for all $m < n$.

For each prime $p \le n$, let $q_p$ be the largest power of $p$ with $q_p \le n$. Consider the equations
\[f(n) \equiv f(n-q_p) \pmod{q_p} \qquad \forall p \le n.\]
Let $Q = \prod_{p \le n} q_p$. By the Chinese Remainder theorem, there exists $y \in \N$ with $y < Q$ so that this system of equations is satisfied if and only if
\[f(n) \equiv y \pmod Q.\]
Our goal is now to show firstly that there exists a product of primes in $R$ that is congruent to $y \pmod Q$, and secondly that this is enough to guarantee $(n-m)|(f(n)-f(m))$ for all $m<n$.

We can decompose $y$ as $y = y_1y_2$, where $(y_1,Q) = 1$ and if a prime $p|y_2$ then $p|Q$, so that $y_1$ is the relatively-prime-to-$Q$ portion of $y$ and $y_2$ shares all of its prime factors with $Q$. We address $y_1$ and $y_2$ separately. Since $(y_1,Q) = 1$, $y_1 \in (\Z/Q\Z)^\times$. By construction, $\mathcal R$ contains a representative of each reduced class modulo $Q$, so there exists some $z_1 \in \mathcal R$ with $z_1 \equiv y_1 \pmod Q$.

As for the $y_2$ portion, it must already be a product of elements of $\mathcal R$. More explicitly, if a prime $p|y_2$, then $p|(y,Q)$. The only equivalence used to construct $y$ with a modulus divisible by $p$ is 
\[f(n) \equiv f(n-q_p) \pmod{q_p},\]
so $p|(y,Q)$ implies $p|f(n-q_p)$. By our induction hypothesis, this implies that $p \in \mathcal R$, so $y_2$ is a product of elements of $\mathcal R$. We can now set $f(n) = z_1 \cdot y_2$, which is then also a product of elements of $\mathcal R$ and equivalent to $y \pmod Q$.

Finally, we can check that $n-m$ divides $f(n)-f(m)$ for all $m<n$. Let $m < n$ and let $p$ be a prime with $p^k||(n-m)$. As before, let $q_p$ be the largest power of $p$ with $q_p \le n$. Note that $p^k$ divides $q_p$, since $p^k \le n-m \le n$. By construction, $f(n) \equiv f(n-q_p) \pmod{q_p}$. Assume that $m \ge n-q_p$; if $m < n-q_p$, the argument is nearly identical. By the inductive hypothesis,
\[f(m) \equiv f(m-(m-n + q_p)) \pmod{m-n+q_p},\]
and since $p^k|(m-n+q_p)$, this implies that $f(m) \equiv f(n-q_p) \pmod{p^k}$, as desired.
\end{proof}

\begin{remark}
As Avi Wigderson pointed out, this construction makes the prime factorization of $f(n)$ explicit. For a prime $p \le n$, let $q_p$ be the largest power of $p$ below $n$, as above, and let $y_p$ be the largest power of $p$ dividing $f(n-q_p)$. Moreover, let $z_1$ be as in our proof of Theorem \ref{main}. Then
\[f(n) = z_1 \cdot \prod_{\substack{p \le n \\ p \in \mathcal R}} y_p.\]
One consequence is that the construction can be tweaked slightly to ensure that $f(n)$ is squarefree for all $n$. If $n$ is the smallest value for which $f(n)$ is not squarefree, then $z_1$ must be one of the primes $p \le n$ in $\mathcal R$, so $f(n)$ being squarefree can be ensured by making $z_1 > n$ for all $n$, for example.
\end{remark}

\begin{proof}[Proof of Corollary \ref{maincor}]
By Theorem \ref{main}, it suffices to construct a set $\mathcal R = \{r_k\}_{k \in \N}$ such that $r_k > b(k)$ for all $k$ and such that for all integers $a$ and $n$ with $(a,n) = 1$, there exists $k$ such that $r_k \equiv a \pmod n$. Let $G = \{(a,n):n \ge 2, 1 \le a \le n, (a,n) = 1\}$ be the set of all pairs $(a,n)$ where $a$ is a reduced residue class modulo $n$. The set $G$ is countable, so we can enumerate it as $G = \{(a_k,n_k)\}_{k \in \N}$. For each $k \in \N$, there are infinitely many primes $r \equiv a_k \pmod{n_k}$, so let $r_k$ be a prime with $r_k \equiv a_k \pmod{n_k}$ and such that $r_k > b(k)$. Then $\mathcal R = \{r_k\}_{k \in \N}$ is precisely the desired set. 
\end{proof}

\section{Proof of Theorem \ref{thm:zannierrho}} \label{sec:zannierrho}

Our proof of Theorem \ref{thm:zannierrho} closely follows the work of Zannier \cite{MR1397665}. Throughout, we fix $0 < \alpha < 1$, a set of primes $\mathcal P$ with density $\alpha$, and a $\alpha$-primary pseudo-polynomial $f$. Recall that for $j \in \N$, we denote by $\Pi(j)$ the quantity $\Pi(j) = \prod_{\substack{p \in \mathcal P \\ p \le j}} p$,
and $\Pi(j)=e^{\a j}+ o(e^{\a j})$.

We begin with the following lemma, which is a consequence of \cite[Lemma 1]{MR733927}. 

\begin{lemma}\label{lem:ANbound}
Let $\{k_j\}_{j \in \N}$ be a sequence of positive real numbers. For $N \in \N$, let
\[A(N) = \left\{ (x_1, \dots, x_N) \in \Z^N : |x_j| \le k_j, x_{n + p} \equiv x_n \pmod p \forall p \in \mathcal P, \forall n \le N - p\right\}.\]
Then
\[\#A(N) \le \prod_{j=1}^N \left(1 + \frac{2k_j}{\Pi(j-1)}\right).\]
\end{lemma}

By assumption, we have $f(n) = O(e^{(2/3 - \ep)\a n})$ and thus $n|f(n)| \le ce^{(2/3)\a n}$ for all $n \in \N$ and for a suitable constant $c$. Set $d = (2/3)\a$, so that $n|f(n)| \le ce^{dn}$. We also fix the constants $\mu = 3/4$ and $\eta = 1/3$. We will need these parameters to satisfy certain inequalities, which we collect here for convenience; the set of inequalities below is not in any sense minimal, but rather precisely what we will use later on. We omit the straightforward verifications of these inequalities.

\begin{lemma}\label{lem:parameters}
The parameters $d = \tfrac 23 \a,$ $\mu = \tfrac 34$, and $\eta = \tfrac 13$ satisfy the following properties, as well as $0 < \mu < 1$.
\begin{enumerate}[label=(\roman*),ref=\theproposition(\roman*)]
	\item \label{lem:paras.i}$\mu d + d -\alpha > 0$
	\item \label{lem:paras.ii}$\mu^2 d > \mu d + \frac{(d-\alpha)}{2}$
	\item \label{lem:paras.iii}$(\mu M + m) d < \alpha\left(\frac m2 + M\right)$ for all $M \in \N$ and all $m \in [2M, (2 + \eta)M]$
	\item \label{lem:paras.iv}$(1 + \mu) d < \alpha(1+\eta)$
\end{enumerate}
\end{lemma}

Let $R$ be a large positive integer to be chosen later. For $\mathbf x = (x_1, \dots, x_R)$ and $\mathbf y = (y_1, \dots, y_R) \in \Z^R$, we define an auxiliary function
\[F_{\mathbf x, \mathbf y}(n) := \sum_{i=1}^R(x_i + y_i n) f(n+i).\]
We will say that a choice of $\mathbf x$ and $\mathbf y$ is \emph{valid} if
\[|x_i| \le e^{d(R-i)}, \qquad |y_i| \le e^{d(R-i)},\]
which implies that for all $n \in \N$,
\[|F_{\mathbf x,\mathbf y}(n)| \le cR e^{d(R+n)}.\]
Let $M_0$ be a large integer so that $R = \mu M_0$ and $R$ is integral; recall that $\mu = \tfrac 34$. 

\begin{lemma}\label{lem:twopairsxy}
There exist two distinct valid choices $(\mathbf x, \mathbf y)$ and $(\mathbf x', \mathbf y')$ with $F_{\mathbf x, \mathbf y}(n) = F_{\mathbf x', \mathbf y'}(n)$ for all $n \in [1,M_0]$.
\end{lemma}
\begin{proof}
Let $k_j = cRe^{d(R+j)}$, so that $|F_{\mathbf x, \mathbf y}(j)| \le k_j$ for all $j$ and all valid choices of $\mathbf x$ and $\mathbf y$.  Each function $F_{\mathbf x, \mathbf y}$ satisfies $F_{\mathbf x, \mathbf y} (n + p) \equiv F_{\mathbf x, \mathbf y}(n) \pmod p$ for all $p \in \mathcal P$ and for all $n \in \N$. With $A(M_0)$ defined as in Lemma \ref{lem:ANbound}, we therefore have that the vector $(F_{\mathbf x,\mathbf y}(1), \dots, F_{\mathbf x, \mathbf y}(M_0)) \in A(M_0)$. Based on our bounds on $x_i$ and $y_i$, we can bound below the number of choices $B$ of $(\mathbf x, \mathbf y)$; we will see from our bound on $B$ and Lemma \ref{lem:ANbound} that there are two distinct elements of $B$ that yield the same element of $A(M_0)$, which will suffice.

The number of choices for each of these $\mathbf x$ and $\mathbf y$ is at least $e^{d R^2/2 + O(R)}$, so the total number of choices for $x_i$ and $y_i$ is at least $e^{d R^2 + O(R)}$. Then
\[\log B \ge d R^2  + O(R) = \mu^2 d M_0^2 + o(M_0^2).\]

Lemma \ref{lem:ANbound} gives that
\[\#A(M_0) \le \prod_{j=1}^{M_0} \left( 1 + \frac{2cRe^{d(R+j)}}{\Pi(j-1)}\right).\]
For all $j \le M_0$, we have
\[\log \left( \frac{2cR e^{d(R+j)}}{\Pi(j-1)}\right) = \mu d M_0 + dj - \alpha j + o(M_0) = \mu d M_0 + (d-\alpha)j + o(M_0).\]
By Lemma \ref{lem:paras.i}, the expression above is positive for large $M_0$, so in particular $\frac{2cR e^{d(R+j)}}{\Pi(j-1)} > 1$ for large $M_0$. But then, summing over $j$,
\begin{align*}
\log \#A(M_0) &\le \sum_{j=1}^{M_0} \log \left(\frac{4cRe^{d(R+j)}}{\Pi(j-1)}\right) \\
&\le \mu d M_0^2 + \frac{(d-\alpha)}{2} M_0^2 + o(M_0^2).
\end{align*}
By Lemma \ref{lem:paras.ii}, we have
\[\mu^2 d > \mu d + \frac{(d-\alpha)}{2},\]
so for large $M_0$ we have $\log B > \log \#A(M_0)$, and thus there are two distinct valid choices of $\mathbf x$ and $\mathbf y$ with
$(F_{\mathbf x,\mathbf y}(1), \dots, F_{\mathbf x,\mathbf y}(n)) = F_{\mathbf x', \mathbf y'}(1),\dots, F_{\mathbf x',\mathbf y'}(n))$, as desired.
\end{proof}

Let $(\mathbf x, \mathbf y)$ and $(\mathbf x', \mathbf y')$ be the pairs furnished by Lemma \ref{lem:twopairsxy}. Define $\Delta:\N \to \Z$ via $\Delta(n) = F_{\mathbf x, \mathbf y}(n) - F_{\mathbf x', \mathbf y'}(n)$. Then $\Delta$ satisfies the following:
\begin{itemize}
	\item $\Delta(n) = \sum_{i=1}^R (\tilde x_i + \tilde y_i n)f(n+i)$, where $\tilde x_i, \tilde y_i \in \Z$, and not all are zero.
	\item $\Delta(n+p) \equiv \Delta(n) \pmod p$ for all $n \in \N$ and all $p \in \mathcal P$.
	\item $|\Delta(n)| \le 2cR e^{d(R+n)}$, and $\Delta(m) = 0$ for $m = 1, 2, \dots, M_0$. 
\end{itemize}

\begin{lemma}
For $M_0$ sufficiently large and $\Delta$ as above, $\Delta(n) = 0$ for all $n \in \N$.
\end{lemma}
\begin{proof}
We will follow the same arguments as in \cite{MR733927} and \cite[pp. 396--397]{MR1397665}.

Assume not. Let $M \in \N$ with $\Delta(m) = 0$ for $m \le M$, and $\Delta(M+1) \ne 0$, so that $M \ge M_0$. We will first show that $\Delta(n) = 0$ for all $n \in I := [2M, (2+\eta)M]$; recall that $\eta = \tfrac 13$. Fix $m \in I$. Let $p \in \mathcal P \cap [1,M]$. Then $\Delta$ vanishes on a complete system of classes $\pmod p$, so $p|\Delta(m)$. 

Now suppose that $p \in \mathcal P \cap [M,m/2)$. Then $0 < m-2p < m-2M \le \eta M < M$, so $\Delta(m) \equiv \Delta(m-2p) \equiv 0 \pmod p$, so $p|\Delta(m)$. Lastly, suppose that $p \in \mathcal P \cap (m-M, m)$. Then $0 < m-p < M$ and $\Delta(m) \equiv \Delta(m-p) \equiv 0 \pmod p$. Since $m-M \ge m/2$, these two sets of primes are distinct.

Thus $\Pi(m/2)\Pi(m)\Pi(m-M)^{-1}$ divides $\Delta(m)$. If $\Delta(m) \ne 0$, then
\[\log |\Delta(m)| \ge \alpha\left(\frac m2 + M \right) + o(M).\]
But, by our constraints on $\Delta$, we have
\[\log |\Delta(m)| \le (R+m)d+ o(M) \le (\mu M + m) d + o(M).\]
By Lemma  \ref{lem:paras.iii}, for all $m \in I$ we have $(\mu M + m) d < \alpha\left(\frac m2 + M\right)$, so $\Delta(m)=0$ if $M_0$ is sufficiently large.

We now show $\Delta(M+1) = 0$, which contradicts our assumption. Since $\Delta(m) = 0$ for $m \in [1,M] \cup [2M,(2+\eta)M]$, $\Delta(M+1)$ is divisible by all the primes $p \in \mathcal P \cap [1,(1+\eta)M]$. Since $\Delta(M+1) \ne 0$, we have
\[\log |\Delta(M+1)| \ge \alpha(1+\eta)M + o(M).\]
But we also have
\[\log|\Delta(M+1)| \le (\mu M + M + 1) d + o(M).\]
By Lemma \ref{lem:paras.iv}, $(1 + \mu) d < \alpha(1+\eta)$, so for large enough $M_0$, we have a contradiction. Thus $\Delta(m) = 0$ for all $m \in \N$.
\end{proof}

We are now ready to show that $f$ is in fact a polynomial. 
\begin{proof}[Proof of Theorem \ref{thm:zannierrho}]
As above, let $\Delta(m) = \sum_{i = 1}^R (\tilde x_i + \tilde y_i m)f(m+i) = 0$ for all $m >0$. Define
\[G(X) = \sum_{m > 0} f(m) X^m,\]
the generating function of our sequence $f$. Since $\sum_{m \ge 0} \Delta(m) X^m = 0$, there is a nontrivial relation
\[a(X) G'(X) + b(X)G(X) + c(X) = 0,\]
where $a,b,c$ are polynomials with integral coefficients and not all zero.

The generating function $G$ must be algebraic, by an argument that is identical to the argument by Zannier in \cite[pp.398--400]{MR1397665}. Since $G$ is algebraic, we can write
\[Q_0(X) G^h(X) + \cdots + Q_h(X) = 0\]
for polynomials $Q_i \in \Z[X]$ with $Q_0 \ne 0$. For all $p \in \mathcal P$, the reduction of $G \pmod p$ is of the form $G_p(X)/(1-X)^p$, for a polynomial $G_p \in \F_p[X]$ of degree $< p$. Write $G(X) \equiv H_p(X)/(1-X)^{e_p} \pmod p$, where $H_p \in \F_p[X]$ has degree $\le e_p$ and is not divisible by $1-X$. We can then take the algebraic equation for $G$ and reduce mod $p$, which shows that $(1-X)^{e_p}$ divides the reduction of $Q_0(X) \in \F_p[X]$; for large $p$, this reduction is not zero. Thus $e_p$ is bounded by some $E$, so $(1-X)^EG(X) \pmod p$ is a polynomial of bounded degree for all large primes $p \in \mathcal P$. Thus $(1-X)^E G(X)$ must in fact be a polynomial. Then $G(X)$ is a rational function whose denominator is a power of $(1-X)$, which implies that $f(m)$ is a polynomial for large $m$. Applying $f(n+p) \equiv f(n) \pmod p$ for large $p \in \mathcal P$ yields that $f$ is a polynomial for all $m>0$.
\end{proof}

\bibliographystyle{amsalpha}

\bibliography{pseudo}









\end{document}